\documentclass[10pt]{article}

\usepackage[margin=1in]{geometry}

\usepackage{mathtools}
\usepackage{amssymb, amsthm, amsfonts}
\usepackage{mathrsfs}

\usepackage{microtype}

\usepackage{enumitem}

\numberwithin{equation}{section}

\theoremstyle{plain}
\newtheorem{theorem}{Theorem}[section]
\newtheorem{lemma}[theorem]{Lemma}
\newtheorem{proposition}[theorem]{Proposition}

\theoremstyle{definition}

\usepackage{titlesec}
\titleformat{\section}
  {\normalfont\large\bfseries\centering}
  {\thesection}{1em}{}
\titleformat{\subsection}
  {\normalfont\normalsize\bfseries}
  {\thesubsection}{1em}{}
\titlespacing*{\section}{0pt}{2.5ex plus 1ex minus .2ex}{1.5ex plus .2ex}
\titlespacing*{\subsection}{0pt}{2.0ex plus 1ex minus .2ex}{1.0ex plus .2ex}

\usepackage[hidelinks]{hyperref}

\title{On the Existence of Integers with at Most 3 Prime Factors Between Every Pair of Consecutive Squares}

\author{Peter J. Campbell}

\date{}

\begin{document}
\maketitle

\begin{center}
\small
\textit{Corresponding author}: Peter J. Campbell (\texttt{p.campbell1@uq.edu.au})\\
\textit{Affiliation}: School of Mathematics and Physics, The University of Queensland, Brisbane, Australia
\end{center}

\begin{abstract}
We prove an explicit analogue of Legendre's conjecture for almost primes. Namely, for every integer \(n \geq 1\), the interval \((n^2,(n+1)^2)\) contains an integer having at most \(3\) prime factors, counted with multiplicity. This improves the previous best result of Dudek and Johnston, who showed that every such interval contains an integer with at most \(4\) prime factors. The proof is divided into two ranges. For \(n^2 \leq 10^{31}\), we use prior computational results on primes in short intervals between consecutive squares, together with explicit bounds on maximal prime gaps. For \(n^2 > 10^{31}\), we give a sieve-theoretic argument with explicit constants, adapting Richert's logarithmic weights to intervals between consecutive squares and employing an explicit linear sieve of Bordignon, Johnston, and Starichkova.
\end{abstract}

\section{Introduction}\label{sec:intro}

A central problem in number theory is to understand how primes are distributed in short intervals. A classical example is Legendre's conjecture: for every integer \(n \geq 1\), the interval \((n^2,(n+1)^2)\) contains a prime. Despite its elementary formulation, this conjecture presents substantial analytic difficulty and remains open even under the Riemann hypothesis. In 1912, Landau included it in his list of open problems in the theory of prime numbers, underscoring its enduring difficulty; see \cite{pintz2009landau} for historical discussion.

Detecting primes in this rigid family of intervals provides a benchmark for current methods in short-interval prime distribution. One way to quantify progress is to relax the problem by enlarging the intervals. For \(\alpha>2\), one asks whether \((n^\alpha,(n+1)^\alpha)\) always contains a prime. After a long sequence of improvements, Baker, Harman, and Pintz showed that this holds for \(\alpha=2.106\) when \(n\) is sufficiently large \cite{baker_harman_pintz2001}.

In this paper we pursue a different relaxation that keeps the interval length but replaces primes with \emph{almost primes}\footnote{Here and throughout, an \emph{almost prime} is a positive integer with a small number of prime factors, counted with multiplicity.} in these intervals. Let \(\Omega(m)\) denote the number of prime factors of \(m\), counted with multiplicity. Thus \(\Omega(m)=1\) if and only if \(m\) is prime, while \(\Omega(m)\leq 2\) means that \(m\) is either prime or the product of two primes (not necessarily distinct). This leads to the following problem: determine how small one can take \(k\) so that for every \(n\geq 1\) there exists an integer \(a\) with
\[
a\in (n^2,(n+1)^2)\quad\text{and}\quad \Omega(a)\leq k.
\]

In 1920, Brun \cite[pp.~24--25]{brun1920crible} showed that for sufficiently large \(n\) one may take \(k=11\). With the development of sieve methods, Chen \cite{chen1975almostprimesinterval} later improved this to \(k=2\) for sufficiently large \(n\). It is much more difficult, however, to obtain an explicit statement that holds for all \(n\geq 1\). Using an explicit form of the linear sieve due to Bordignon, Johnston, and Starichkova~\cite{bordignon_johnston_starichkova2025}, Dudek and Johnston~\cite[Theorem~1.2]{dudek_johnston2026apba} proved that for every integer \(n\geq 1\) there exists \(a\in(n^2,(n+1)^2)\) such that \(\Omega(a)\leq 4\). Their argument is based on an explicit weighted sieve of Kuhn~\cite{kuhn1954neue_abschaetzungen} together with computational verification in the small-\(n\) range.

A closely related framework was used by Johnston, Sorenson, Thomas, and Webster~\cite{johnston2026primesprimescubes} for intervals between consecutive cubes. Starting from the same explicit-sieve backbone, they replace Kuhn's weights with Richert's logarithmic weights~\cite{richert1969selberg_weights} and prove that for every \(n\geq 1\) there exists \(a\in(n^3,(n+1)^3)\) with \(\Omega(a)\leq 2\). They report that Richert's weights give much better numerical performance than Kuhn's weights. The square-interval problem is naturally more difficult computationally than the cube-interval problem, since the intervals \((n^2,(n+1)^2)\) are substantially shorter than \((n^3,(n+1)^3)\), leaving less flexibility in the finite verification needed to bridge the small-\(n\) and large-\(n\) regimes.

Motivated by the numerical improvements obtained from Richert's logarithmic weights in the cube setting, we adapt the methods of Johnston, Sorenson, Thomas, and Webster~\cite{johnston2026primesprimescubes} to the square-interval problem considered by Dudek and Johnston. Our result improves the bound on the number of prime factors.
\begin{theorem}\label{thm:main_omega3}
    For every integer \(n\geq 1\), there exists \(a\in(n^2,(n+1)^2)\) such that \(\Omega(a)\leq 3\).
\end{theorem}

The main novel ingredient in our proof is the small-\(n\) verification, Lemma~\ref{lem:small_n}. This substantially extends the finite range covered by Dudek and Johnston~\cite{dudek_johnston2026apba}. The lemma combines the verification of Legendre's conjecture up to \(n\leq 7.05\cdot 10^{13}\) by Sorenson and Webster~\cite{sorenson_webster_2025_legendre} with explicit computations of maximal prime gaps~\cite{pzktupel_risinggap}. For the remaining values of \(n\) in the finite range, we use the prime-gap computations to construct a semiprime in \((n^2,(n+1)^2)\). This proves the stronger bound \(\Omega(a)\leq 2\) for all \(n\) with \(n^2\leq 10^{31}\), leaving only the range \(n^2>10^{31}\) for the analytic argument.

For the remaining range, the large-\(n\) analysis follows the explicit framework of Johnston, Sorenson, Thomas, and Webster~\cite{johnston2026primesprimescubes} closely. We include many of the underlying lemmas to keep the exposition self-contained, and we supply full proofs where modifications are required in the square-interval setting. For convenience, we describe these modifications using notation introduced later in the paper. All such notation is defined at the relevant point in the argument. Where practical, we also retain the notation of Johnston, Sorenson, Thomas, and Webster to facilitate comparison with the cube-interval paper. Specifically, the transition from cubes to squares requires the following changes:
\begin{enumerate}[label=\roman*.]
    \item Lemma~\ref{lem:small_n}: the finite verification just described, which covers all \(n\) with \(n^2 \leq 10^{31}\) and thereby reduces the analytic argument to the range \(n^2 > 10^{31}\).
    \item Lemma~\ref{lem:mertens_explicit_bounds}: an explicit admissible choice of the parameter \(\varepsilon\) for the range of \(z\) arising from our final parameter choices, enabling sharper numerical constants in the final optimisation step.
    \item Proposition~\ref{prop:weighted_sifting_bound}: an explicit upper bound for the weighted sifting sum in \eqref{eq:Sap_sum}, tailored to the square-interval set \(\mathcal A(N)=\mathbb Z\cap (N,N+2\sqrt N)\), where \(N=n^2\) in the notation above. The sieve decomposition follows the cube-interval argument, but the shorter square intervals change the natural scale: the size of \(\mathcal A(N)\) is on the order of \(\sqrt N\). Consequently, in the final optimisation step, we bound all contributions by explicit multiples of \(\sqrt N/\log X\) rather than \(N^{2/3}/\log X\). The exponent \(2/3\) arising in the cube setting is therefore replaced by the exponent \(1/2\) throughout; it appears in the admissible range for \(\alpha\), in the cutoff \(w\), and in the definition of \(D_p\). We also replace \(3N^{2/3}\) by \(2\sqrt N\), reflecting the change in interval length in the squares setting. Finally, we track a slightly sharper bound in the \(M_1(X)\) contribution than the one recorded in \cite{johnston2026primesprimescubes}, since this is numerically relevant for our final choice of parameters.
    \item Lemma~\ref{lem:Q0_omega3}: an explicit verification of the square-factor condition \eqref{eq:Q0_cond} for our square-interval set \(\mathcal A(N)\), giving concrete constants \(c\) and \(\delta\) for the parameter choices used in the proof. This controls the loss in passing from \(\mathcal A\) to \(\mathcal A'\) in \eqref{eq:A_prime_def} by bounding \(\sum_{z\le p<y}|\mathcal A_{p^2}|\) via the tail estimate \(\sum_{p\ge z}p^{-2}\) and an explicit upper bound for \(\pi(y)\).
    \item Theorem~\ref{thm:main_omega3}: the final optimisation step, where we choose parameters specific to the square setting to deduce the main result.
\end{enumerate}

A natural next goal would be to replace the bound \(\Omega(a)\leq 3\) by \(\Omega(a)\leq 2\). This is beyond the present framework. Weighted sieves are subject to a natural limitation in this direction, illustrated by the section ``An Extremal Example for the Weighted Sieve'' in Chapter~5 of Greaves~\cite{greaves2001sieves}. Results producing \(2\)-almost primes in related problems have often required more flexible sieve input, such as Iwaniec's form of the error term in the linear sieve~\cite{iwaniec1980new}, which permits the incorporation of bilinear estimates. We are not aware of an explicit version of this form of the linear sieve that is suitable for the present problem. Even with such an explicit version, obtaining an \(\Omega(a)\leq 2\) result for all \(n\geq 1\) would require substantial further optimisation beyond the methods used here.

The paper is organised as follows. Section~\ref{sec:prelims} establishes sieve notation and explicit sieve bounds, and includes the computational verification for small \(n\) needed for Theorem~\ref{thm:main_omega3}. Section~\ref{sec:Richertsect} sets up Richert's weighted sieve in the present context. Section~\ref{sec:mainsect} optimises the parameters and completes the proof of Theorem~\ref{thm:main_omega3}.

\section{Preliminaries of sieving}\label{sec:prelims}

We begin by setting up the sieve notation used throughout. Let \(\mathcal{A}\) be a finite set of integers and let \(\mathcal{P}\) be an infinite set of primes (in our application, \(\mathcal{P}\) is the set of all primes). Given \(z\geq 2\), the basic aim is to sieve (``sift'' out) \(\mathcal{A}\) by removing those elements divisible by at least one prime \(p\in\mathcal{P}\) with \(p<z\). In our application we consider the short interval
\begin{equation}\label{Adef}
    \mathcal{A}=\mathcal{A}(N):=\mathbb{Z}\cap (N,\,N+2\sqrt{N}).
\end{equation}

If we set \(N=n^2\), then \(N+2\sqrt{N}=(n+1)^2-1\), and hence \(\mathcal{A}(n^2)\subset (n^2,(n+1)^2)\). Therefore, to prove Theorem~\ref{thm:main_omega3} it suffices to show that for each \(N=n^2\) there exists \(a\in\mathcal{A}(N)\) with \(\Omega(a)\leq 3\).

For our purposes, we may assume \(N \geq 10^{31}\), since the remaining range can be handled by computation. The following lemma strengthens the small-\(n\) verification appearing in Dudek and Johnston~\cite[Lemma~2.1]{dudek_johnston2026apba} by proving that for every integer \(n\geq 1\) with \(n^2\leq 10^{31}\), the interval \((n^2,(n+1)^2)\) contains an integer with at most two prime factors. Sorenson and Webster verified that for every \(n\leq 7.05\cdot 10^{13}\), the interval \((n^2,(n+1)^2)\) contains a prime \cite{sorenson_webster_2025_legendre}. To extend the verification further, we also use record computations on \emph{maximal prime gaps}\footnote{Let \(p_n\) denote the \(n\)th prime and write \(g_n:=p_{n+1}-p_n\). We call \(g_n\) a \emph{maximal prime gap} if \(g_m<g_n\) for all \(m<n\).}. The current record computations imply that every gap between consecutive primes below \(6.8\cdot 10^{19}\) has size at most \(1724\) \cite{pzktupel_risinggap}. Combining these inputs yields the following lemma.

\begin{lemma}\label{lem:small_n}
    For every integer \(n\geq 1\) with \(n^2\leq 10^{31}\), there exists \(a \in (n^2,(n+1)^2)\) such that \(\Omega(a) \leq 2\).
\end{lemma}

\begin{proof}
    Let \(n\geq 1\) with \(n^2\leq 10^{31}\). If \(n\leq 7.05\cdot 10^{13}\), then Sorenson and Webster verified that \((n^2,(n+1)^2)\) contains a prime in this range \cite{sorenson_webster_2025_legendre}.
    
    Now assume \(n>7.05\cdot 10^{13}\). To cover this range, we construct a semiprime \(pq\in(n^2,(n+1)^2)\), where \(p\) and \(q\) are prime. The role of \(p\) is to ensure that the rescaled interval for \(q\) has length exceeding the known prime-gap bound while remaining below \(6.8\cdot 10^{19}\), so that the current record computations on maximal prime gaps may be applied. Indeed, by \cite{pzktupel_risinggap}, every gap between consecutive primes with starting prime below \(6.8\cdot 10^{19}\) has size at most \(1724\). Equivalently, for every positive real \(x<6.8\cdot 10^{19}\), there exists a prime \(q\in(x,x+1724]\).
    
    \smallskip
    \noindent\textbf{Case 1: \(n^2<10^{30}\).}
    Let \(p = 14706000011\) (which is prime), and consider the interval
    \[
    I:=\left(\frac{n^2}{p},\,\frac{(n+1)^2}{p}\right).
    \]
    Its length is
    \[
    |I|=\frac{(n+1)^2-n^2}{p}=\frac{2n+1}{p}
    \geq \frac{2\cdot 7.05\cdot 10^{13}+1}{14706000011}>1724.
    \]
    Also, since \(n<10^{15}\) in this case,
    \[
    \frac{(n+1)^2}{p}\leq \frac{(10^{15}+1)^2}{14706000011}
    < 6.8\cdot 10^{19}.
    \]
    Therefore the prime-gap bound applies at \(x=n^2/p\), and since \(|I|>1724\) there exists a prime \(q\in I\). Thus, we have
    \[
    n^2<pq<(n+1)^2.
    \]
    
    \smallskip
    \noindent\textbf{Case 2: \(10^{30}\leq n^2\leq 10^{31}\).}
    Let \(p = 147058823551\) (which is prime), and again consider
    \[
    I:=\left(\frac{n^2}{p},\,\frac{(n+1)^2}{p}\right).
    \]
    Now \(n\geq 10^{15}\), so
    \[
    |I|=\frac{2n+1}{p} \geq \frac{2\cdot 10^{15}+1}{147058823551} > 1724.
    \]
    Moreover,
    \[
    \frac{(n+1)^2}{p}
    \leq \frac{10^{31}+2\sqrt{10^{31}}+1}{147058823551}
    < 6.8\cdot 10^{19}.
    \]
    Therefore the prime-gap bound again applies at \(x=n^2/p\), and since \(|I|>1724\) there exists a prime \(q\in I\). Thus, we have
    \[
    n^2<pq<(n+1)^2.
    \]
    
    Combining the cases proves the lemma.
\end{proof}

The argument in Lemma~\ref{lem:small_n} can be extended beyond the stated range by choosing an appropriate prime \(p\) so that the corresponding interval for \(q\) stays within the range covered by the known prime-gap bound, while still having length exceeding that bound. We do not pursue this here, since the range \(n^2 \leq 10^{31}\) is sufficient for our purposes.

We now return to the sieve notation required for the analytic part of the argument. Given \(z\geq 2\), we quantify the set of remaining elements of \(\mathcal{A}\) after \(\mathcal{A}\) has been sifted by primes less than \(z\) as
\begin{equation}\label{sfunctiondef}
S(\mathcal{A},\mathcal{P},z)
:= \left| \mathcal{A} \setminus \bigcup_{\substack{p \in \mathcal{P} \\ p < z}} \mathcal{A}_p \right|,
\end{equation}
where
\begin{equation}\label{pzdef}
\mathcal{A}_p := \{a\in\mathcal{A} : p\mid a\},
\quad \text{and} \quad
P(z) := \prod_{\substack{p\in\mathcal{P}\\ p<z}} p.
\end{equation}
Thus \(S(\mathcal{A},\mathcal{P},z)\) counts the elements of \(\mathcal{A}\) free of prime divisors \(p\in\mathcal{P}\) with \(p<z\).

\subsection{An explicit linear sieve}

To detect integers in \(\mathcal{A}\) with few prime factors, we require effective upper and lower bounds for \(S(\mathcal{A},\mathcal{P},z)\). For this purpose we invoke the explicit linear sieve of Bordignon, Johnston, and Starichkova~\cite{bordignon_johnston_starichkova2025}. Rather than quoting their full general statement, we record below a streamlined form sufficient for our applications, following the version used in~\cite{dudek_johnston2026apba,johnston2026primesprimescubes}. Throughout this paper, \(\gamma\) denotes the Euler--Mascheroni constant.
\begin{lemma}[{Explicit version of the linear sieve \cite[Theorem~2.1]{bordignon_johnston_starichkova2025}}]\label{lem:linear_sieve}
    Let \(\mathcal{A}\) be a finite set of positive integers and let \(\mathcal{P}\) be a set of primes such that no element of \(\mathcal{A}\) is divisible by any prime outside \(\mathcal{P}\).
    Let \(P(z)\) and \(S(\mathcal{A},\mathcal{P},z)\) be defined as in \eqref{pzdef} and \eqref{sfunctiondef}.
    For each square-free integer \(d\), let \(g(d)\) be a multiplicative function satisfying \(0\leq g(p)<1\) for all \(p\in\mathcal{P}\), and define
    \[
    \mathcal{A}_d:=\{a\in\mathcal{A}: d\mid a\},\qquad
    r(d):=|\mathcal{A}_d|-|\mathcal{A}|\,g(d).
    \]
    Let \(\mathcal{Q}\subset\mathcal{P}\) and write \(Q:=\prod_{q\in\mathcal{Q}}q\).
    Assume there exists \(\varepsilon\) with \(0<\varepsilon\leq 1/74\) such that, for all \(1<u<z\),
    \begin{equation}\label{eq:eps_condition}
    \prod_{\substack{p\in\mathcal{P}\setminus\mathcal{Q}\\ u\leq p<z}}(1-g(p))^{-1}
    <(1+\varepsilon)\frac{\log z}{\log u}.
    \end{equation}
    Then for any \(D\geq z\) we have the upper bound
    \begin{equation}\label{eq:sieve_upper}
    S(\mathcal{A},\mathcal{P},z)
    <|\mathcal{A}|\,V(z)\left(F(s)+\varepsilon C_1(\varepsilon)e^2h(s)\right)+R(D),
    \end{equation}
    and for any \(D\geq z^2\) we have the lower bound
    \begin{equation}\label{eq:sieve_lower}
    S(\mathcal{A},\mathcal{P},z)
    >|\mathcal{A}|\,V(z)\left(f(s)-\varepsilon C_2(\varepsilon)e^2h(s)\right)-R(D),
    \end{equation}
    where
    \[
    s:=\frac{\log D}{\log z},
    \]
    \begin{equation}\label{eq:h_def}
    h(s):=\begin{cases}
    s^{-1}e^{-2}, &0< s\leq 1,\\
    e^{-2}, &1\leq s\leq 2,\\
    e^{-s}, &2\leq s\leq 3,\\
    3s^{-1}e^{-s}, &s\geq 3,
    \end{cases}
    \end{equation}
    and \(F(s)\), \(f(s)\) are defined by the delay differential system
    \begin{equation}\label{eq:Ff_def}
    \begin{cases}
    F(s)=\dfrac{2e^{\gamma}}{s},\quad f(s)=0 & \text{for } 0<s\leq 2,\\[0.6em]
    (sF(s))'=f(s-1),\quad (sf(s))'=F(s-1) & \text{for } s\geq 2.
    \end{cases}
    \end{equation}
    Moreover,
    \[
    V(z):=\prod_{\substack{p\in\mathcal{P}\\ p<z}}(1-g(p)),
    \qquad
    R(D):=\sum_{\substack{d\mid P(z)\\ d<QD}}|r(d)|,
    \]
    and \(C_1(\varepsilon)\), \(C_2(\varepsilon)\) are as in \cite[Table~1]{bordignon_johnston_starichkova2025}.
\end{lemma}

The definition of \(h(s)\) in \eqref{eq:h_def} for \(0 < s \leq 1\), namely \(h(s)=s^{-1}e^{-2}\), is a convenient extension of the notation used in \cite[Theorem~2.1]{bordignon_johnston_starichkova2025}. This choice is not part of the original statement, but it is useful for our application because we will require a weak upper-bound estimate for \(S(\mathcal{A},\mathcal{P},z)\) in the case \(D<z\) (equivalently, \(s < 1\)). We follow the same convention as in \cite{johnston2026primesprimescubes}.

For later use, we record the explicit formulas for \(F(s)\) and \(f(s)\) obtained from the delay differential system \eqref{eq:Ff_def} in the range relevant to our parameter choices. Although more complicated expressions are available for larger values of \(s\) (see, for example, \cite{cai2008chen2}), the following formulas are sufficient for our purposes.

\begin{lemma}\label{lem:Ff_explicit}
    Let \(F(s)\) and \(f(s)\) be defined by \eqref{eq:Ff_def}. Then
    \[
    F(s)=\frac{2e^\gamma}{s}\qquad (0<s\leq 3),
    \]
    and
    \[
    f(s)=\frac{2e^\gamma\log(s-1)}{s}\qquad (2\leq s \leq 4).
    \]
\end{lemma}

To keep the remainder term \(R(D)\) small, we choose \(g(d)\) so that \(|\mathcal{A}_d|\) is well-approximated by \(|\mathcal{A}|\,g(d)\). Since \(\mathcal{A}\) is a set of consecutive integers, we take \(g(d)=1/d\), which gives
\[
r(d)=|\mathcal{A}_d|-\frac{|\mathcal{A}|}{d}.
\]
With this choice,
\[
V(z)=\prod_{\substack{p\in\mathcal{P}\\ p<z}}\left(1-\frac{1}{p}\right).
\]
To estimate the main term in Lemma~\ref{lem:linear_sieve}, we require explicit bounds for \(V(z)\), and for this we use the following estimates.
\begin{lemma}[{\cite[Theorem~7]{rosser_schoenfeld1962}}]\label{lem:mertens_products}
    \begin{equation}\label{eq:mertens_lower}
    \frac{e^{-\gamma}}{\log z}\left(1-\frac{1}{2\log^2 z}\right) < \prod_{p<z}\left(1-\frac{1}{p}\right)
    \qquad (z\geq 285),
    \end{equation}
    and
    \begin{equation}\label{eq:mertens_upper}
    \prod_{p<z}\left(1-\frac{1}{p}\right) < \frac{e^{-\gamma}}{\log z}\left(1+\frac{1}{2\log^2 z}\right)
    \qquad (z>1).
    \end{equation}
\end{lemma}

To choose admissible values of \(\varepsilon\) in Lemma~\ref{lem:linear_sieve}, we require an explicit bound for the prime product in \eqref{eq:eps_condition} (a partial Mertens product). The following estimate is sufficient for the ranges relevant to our argument.

\begin{lemma}\label{lem:mertens_explicit_bounds}
    For all \(u\geq 3\) and \(z \geq 7080\), we have
    \begin{equation}\label{eq:partial_mertens_1}
    \prod_{u\leq p< z}\left(1-\frac{1}{p}\right)^{-1} < \left(1 + 1.13 \cdot 10^{-3}\right)\frac{\log z}{\log u}.
    \end{equation}
\end{lemma}

\begin{proof}
    We argue by ranges of \(z\).
    
    \smallskip
    \noindent\textbf{Case 1: \(7080\leq z<10^5\).}
    For this range we checked \eqref{eq:partial_mertens_1} for all \(3\leq u\leq z\) by direct computation, using Johnston's script~\cite{mertensbounds_github}; this is the same script used for the verification in~\cite[Lemma~3.4]{johnston2026primesprimescubes}.
    
    \smallskip
    \noindent\textbf{Case 2: \(z\geq 10^5\).}
    A stronger bound is proved in \cite[Lemma~3.4]{johnston2026primesprimescubes}, namely
    \[
    \prod_{u\leq p<z}\left(1-\frac{1}{p}\right)^{-1}
    <
    \left(1+2.8\cdot 10^{-4}\right)\frac{\log z}{\log u}
    \qquad (u\geq 3,\ z\geq 10^5).
    \]
    Since \(2.8\cdot 10^{-4}<1.13\cdot 10^{-3}\), this implies \eqref{eq:partial_mertens_1} for all \(u\geq 3\) and \(z\geq 10^5\).
    
    This completes the proof.
\end{proof}

In Lemma~\ref{lem:linear_sieve} we will take \(\mathcal Q=\{2\}\), so \(Q=2\). Then \eqref{eq:partial_mertens_1} verifies \eqref{eq:eps_condition} with \(\varepsilon=1.13\cdot 10^{-3}\) for \(z\geq 7080\).

\section{Richert's weighted sieve}\label{sec:Richertsect}

We now introduce Richert's logarithmic sieve weights, which will be used to obtain a lower bound for
\begin{equation}\label{eq:rk_def}
r_k(\mathcal{A}) := \left|\{a\in \mathcal{A} : \Omega(a)\leq k\}\right|.
\end{equation}
Throughout this section we work in generality, and the statements hold for an arbitrary finite set \(\mathcal A\), not only for the specific choice \(\mathcal A(N)\) from Section~\ref{sec:prelims}.

Let \(\mathcal P\) be a set of primes, and let \(P(z)\) be as in \eqref{pzdef}. Fix \(k\in\mathbb N\) as in \eqref{eq:rk_def}. Let \(k_1\) and \(k_2\) be real numbers with \(k_1\geq k_2\geq 1\) and \(k_2<k+1\). Write
\begin{equation}\label{eq:X_def}
X := \max(\mathcal{A}),
\end{equation}
\begin{equation}\label{eq:zy_def}
z := X^{1/k_1}, \quad y := X^{1/k_2}.
\end{equation}

For \(a\in\mathcal{A}\) with \((a,P(z))=1\), define Richert's logarithmic weight
\[
w(a)
:= \lambda-\sum_{\substack{p\in\mathcal{P}\\ z\leq p<y\\ p\mid a}}
\left(1-\frac{\log p}{\log y}\right),
\]
where
\begin{equation}\label{eq:lambda_def}
\lambda:=k+1-k_2.
\end{equation}
We then define the weighted sifting function
\begin{equation}\label{eq:W_def}
W(\mathcal{A},\mathcal{P},z)
=W(\mathcal{A},\mathcal{P},z,k_1,k_2,\lambda)
:=\sum_{\substack{a\in\mathcal{A}\\ (a,P(z))=1}} w(a).
\end{equation}

We first relate \(W(\mathcal A,\mathcal P,z)\) to \(r_k(\mathcal A)\). For this it is convenient to exclude elements with square factors in the intermediate prime range. Define
\begin{equation}\label{eq:A_prime_def}
\mathcal{A}'
:=
\mathcal{A}\setminus
\bigcup_{\substack{p\in\mathcal{P}\\ z\leq p<y}}
\mathcal{A}_{p^2}.
\end{equation}
Thus \(\mathcal{A}'\) consists of those elements of \(\mathcal{A}\) with no square divisor \(p^2\) coming from primes \(p\in\mathcal{P}\) with \(z\leq p<y\). In typical applications, the quantity
\[
\sum_{\substack{p\in\mathcal{P}\\ z\leq p<y}} |\mathcal{A}_{p^2}|
\]
is negligible compared with \(|\mathcal{A}|\).

The next lemma shows that a lower bound for \eqref{eq:W_def}, applied to \(\mathcal A'\), yields a lower bound for \(r_k(\mathcal A)\).

\begin{lemma}[{\cite[Lemma~3.5]{johnston2026primesprimescubes}}]\label{lem:W_to_rk}
    With \(r_k(\mathcal{A})\) as in \eqref{eq:rk_def}, \(W(\mathcal{A},\mathcal{P},z)\) as in \eqref{eq:W_def}, and \(\mathcal{A}'\) as in \eqref{eq:A_prime_def}, we have
    \[
    r_k(\mathcal{A}) \geq \frac{1}{k}\,W(\mathcal{A}',\mathcal{P},z).
    \]
\end{lemma}

Next we express \(W(\mathcal A',\mathcal P,z)\) in terms of ordinary sifting functions, so that Lemma~\ref{lem:linear_sieve} may be applied. To control the loss incurred by passing from \(\mathcal{A}\) to \(\mathcal{A}'\), we assume the following condition.

\begin{equation}\tag{$Q_0(c,\delta)$}\label{eq:Q0_cond}
\sum_{\substack{p\in\mathcal{P}\\ z\leq p<y}} |\mathcal{A}_{p^2}|
\leq c\,|\mathcal{A}|^{1-\delta},
\end{equation}
for some constants \(c>0\) and \(0<\delta<1\). In particular, \eqref{eq:Q0_cond} implies that \(|\mathcal{A}\setminus\mathcal{A}'|\) is small.

\begin{lemma}[{\cite[Lemma~3.6]{johnston2026primesprimescubes}}]\label{lem:rk_lower_weighted}
    Assume \eqref{eq:Q0_cond}. Then
    \begin{align}\label{eq:rk_lower_main}
    r_k(\mathcal{A})
    \geq\;& \frac{\lambda}{k}\,S(\mathcal{A},\mathcal{P},z)
    -\frac{1}{k}\sum_{\substack{p\in\mathcal{P}\\ z\leq p<y}} \left(1-\frac{\log p}{\log y}\right)S(\mathcal{A}_p,\mathcal{P},z)
    -\frac{\lambda c}{k}\,|\mathcal{A}|^{1-\delta},
    \end{align}
    where \(\lambda\) is defined by \eqref{eq:lambda_def}.
\end{lemma}

The right-hand side of \eqref{eq:rk_lower_main} is now suitable for lower bounds via Lemma~\ref{lem:linear_sieve}, applied to \(\mathcal A\) and to the sets \(\mathcal A_p\).

\subsection{An upper bound for the weighted sifting term}

By Lemma~\ref{lem:rk_lower_weighted}, it remains to obtain an explicit upper bound for the weighted sifting term
\begin{equation}\label{eq:Sap_sum}
\sum_{\substack{p\in\mathcal{P}\\ z\leq p<y}}
\left(1-\frac{\log p}{\log y}\right)S(\mathcal{A}_p,\mathcal{P},z).
\end{equation}
Bounding \eqref{eq:Sap_sum} is the main task of the remainder of this section. We begin by recording two auxiliary lemmas from the literature.

\begin{lemma}[{\cite[\S1.3.5, Lemma~1(ii)]{greaves2001sieves}}]\label{lem:partial_summation_upper}
    Let \(f(t)\) be a positive monotone function on \([z,y]\), with \(f'(t)\) piecewise continuous on \([z,y]\), and let \(c(n)\) be an arithmetic function. Suppose that
    \[
    \sum_{x \leq n<w} c(n) \leq g(w)-g(x)+E
    \]
    for some constant \(E\), whenever \(z\leq x<w\leq y\), where \(g\) is differentiable on \([z,y]\). Then
    \[
    \sum_{z\leq n<y} c(n)f(n)
    \leq
    \int_z^y f(t)g'(t)\,dt
    + E\max\{f(z),f(y)\}.
    \]
\end{lemma}

We also require an explicit estimate for reciprocal prime sums over intervals, which we quote in the following form.

\begin{lemma}[{\cite[Corollary~1]{vanlalngaia2017explicit}}]\label{lem:vanlalngaia_prime_reciprocals}
    For any real numbers \(b>a>1000\), we have
    \[
    \sum_{a\leq p<b}\frac{1}{p}
    <
    \log\log b-\log\log a+\frac{5}{(\log a)^3}.
    \]
\end{lemma}

The next result packages the linear-sieve upper bound into the weighted form required by \eqref{eq:Sap_sum}. It is the square-interval analogue of the corresponding estimate in \cite[Proposition 3.9]{johnston2026primesprimescubes}. The only inputs are the upper-bound linear sieve (Lemma~\ref{lem:linear_sieve}), an explicit Mertens product bound (Lemma~\ref{lem:mertens_products}), and explicit control of weighted reciprocal prime sums via partial summation. The dependence on \(k_1,k_2,\alpha,\varepsilon\) and \(Q\) is left explicit for the optimisation in Section~\ref{sec:mainsect}.

\begin{proposition}\label{prop:weighted_sifting_bound}
    Let \(\mathcal A=\mathcal A(N)=\mathbb Z\cap (N,N+2\sqrt N)\) be as in \eqref{Adef}. Let \(\mathcal{P}\) denote the set of all primes, and let \(X,z,y\) be as in \eqref{eq:X_def} and \eqref{eq:zy_def}. Assume that \(y>z>1000\) and \(k_1\leq 8\). Choose a real parameter \(\alpha\) with
    \begin{equation}\label{eq:alpha_range}
    0<\alpha<\frac{1}{2}-\frac{1}{k_2},
    \end{equation}
    and put
    \begin{equation}\label{eq:kalpha_def}
    k_\alpha:=k_1\!\left(\frac12-\frac{1}{k_2}-\alpha\right).
    \end{equation}
    For each prime \(p\), define
    \begin{equation}\label{eq:Dp_def}
    D_p:=\frac{X^{\frac12-\alpha}}{p},
    \end{equation}
    \begin{equation}\label{eq:sp_def}
    s_p:=\frac{\log D_p}{\log z},
    \end{equation}
    and set
    \begin{equation}\label{eq:Dtilde_def}
    \widetilde D:=\min\{D_y,z\}.
    \end{equation}
    Let \(\mathcal{Q}\) be a set of primes, write \(Q:=\prod_{q\in\mathcal{Q}}q\), and let \(\varepsilon>0\) satisfy
    \begin{equation}\label{eq:eps_condition_prop}
    \prod_{\substack{\ell\in\mathcal{P}\setminus \mathcal{Q}\\ u\leq \ell<x}}
    \left(1-\frac{1}{\ell}\right)^{-1}
    <(1+\varepsilon)\frac{\log x}{\log u},
    \end{equation}
    for all \(x\geq \widetilde D\) and \(1<u\leq x\). Then
    \begin{equation}\label{eq:weighted_sum_prop}
    \sum_{\substack{p\in\mathcal{P}\\ z\leq p<y}}
    \left(1-\frac{\log p}{\log y}\right) S(\mathcal{A}_p,\mathcal{P},z)
    < k_1 e^{-\gamma}\!\left(1+\frac{1}{2(\log \widetilde D)^2}\right)
    \left(M_1(X)+M_2(X)\right)+\mathcal E(X),
    \end{equation}
    where
    \begin{align}
    M_1(X)
    :=\;& \frac{2\sqrt{N}}{\log X}\Biggl[
    \frac{2e^\gamma}{k_1}
    \Biggl(
    \frac{1}{1-2\alpha}
    \left(
    2 \log\frac{k_1}{k_2}
    +\left(k_2(1 - 2\alpha)-2\right)
    \log\frac{1- 2\alpha-\frac{2}{k_2}}{1-2\alpha-\frac{2}{k_1}}
    \right)
    \nonumber\\
    &\qquad\quad
    +\frac{5k_1^3}{\left(\frac{1}{2} - \alpha - \frac{1}{k_1}\right)(\log X)^3}\left(1-\frac{k_2}{k_1}\right)\Biggr)
    \nonumber\\
    &\qquad\qquad
    +\varepsilon C_1(\varepsilon)e^2 h(k_\alpha)
    \Biggl(
    \log\frac{k_1}{k_2}-1+\frac{k_2}{k_1}
    +\frac{5k_1^3}{(\log X)^3}\left(1-\frac{k_2}{k_1}\right)
    \Biggr)
    \Biggr], \label{eq:M1_def}
    \end{align}
    \begin{equation}\label{eq:M2_def}
    M_2(X)
    := \frac{y}{\log X}\left(1-\frac{k_2}{k_1}\right)
    \left(\frac{2e^\gamma}{k_\alpha}+\varepsilon C_1(\varepsilon)e^2 h(k_\alpha)\right),
    \end{equation}
    and
    \begin{equation}\label{eq:E_def}
    \mathcal E(X)
    := Q X^{\frac{1}{2} - \alpha}
    \left(
    \log\frac{k_1}{k_2}-1+\frac{k_2}{k_1}
    +\frac{5k_1^3}{(\log X)^3}\left(1-\frac{k_2}{k_1}\right)
    \right).
    \end{equation}
    Here \(C_1(\varepsilon)\) and \(h(s)\) are as in Lemma~\ref{lem:linear_sieve}.
\end{proposition}

\begin{proof}
    Let \(p\) be a prime with \(z\leq p<y\). Our first task is to obtain an upper bound for \(S(\mathcal A_p,\mathcal P,z)\) that is uniform in \(p\). The role of \(D_p=X^{\frac12-\alpha}/p\) is that the linear sieve lemma naturally produces estimates at level \(D_p\), depending on whether \(D_p\) lies above or below the sifting level \(z\). Define
    \[
    w:=X^{\frac{1}{2}-\alpha-\frac{1}{k_1}}.
    \]
    We treat separately the ranges \(z\leq p\leq w\) and \(w<p<y\).
    
    \smallskip
    \noindent\textbf{Case 1: \(z\leq p\leq w\).}
    Here \(D_p\geq z\). We invoke the upper-bound inequality \eqref{eq:sieve_upper} from Lemma~\ref{lem:linear_sieve} with \(g(d)=1/d\), \(D=D_p\), and \(s=s_p\), obtaining
    \begin{equation}\label{eq:Sq_case1}
    S(\mathcal{A}_p, \mathcal{P}, z)
    < |\mathcal{A}_p| V(z)\left(F(s_p)+\varepsilon C_1(\varepsilon)e^2h(s_p)\right)
    +\sum_{\substack{d\mid P(z)\\ d < Q D_p}} |r_p(d)|,
    \end{equation}
    where
    \begin{equation}\label{eq:rq_def}
    r_p(d):=|\mathcal{A}_{pd}|-\frac{|\mathcal{A}_p|}{d}.
    \end{equation}
    We now bound \(V(z)\) using Lemma~\ref{lem:mertens_products} as
    \[
    V(z)=\prod_{\substack{\ell \in \mathcal P \\ \ell < z}}\left(1-\frac{1}{\ell}\right)
    \leq \frac{e^{-\gamma}}{\log z}\left(1+\frac{1}{2(\log z)^2}\right)
    \leq \frac{k_1e^{-\gamma}}{\log X}\left(1+\frac{1}{2(\log \widetilde D)^2}\right).
    \]
    Substituting this into \eqref{eq:Sq_case1} gives
    \begin{equation}\label{eq:Sq_case1_refined}
    S(\mathcal{A}_p,\mathcal{P},z)
    < k_1e^{-\gamma}\left(1+\frac{1}{2(\log \widetilde D)^2}\right)
    |\mathcal{A}_p|
    \left(\frac{F(s_p)+\varepsilon C_1(\varepsilon)e^2 h(s_p)}{\log X}\right)
    +\sum_{\substack{d\mid P(z)\\ d< Q D_p}} |r_p(d)|.
    \end{equation}
    
    \smallskip
    \noindent\textbf{Case 2: \(w<p<y\).}
    In this range, \(D_p<z\). We first lower the sieve level using monotonicity in the sifting parameter. This gives us
    \[
    S(\mathcal{A}_p,\mathcal{P},z)\leq S(\mathcal{A}_p,\mathcal{P},D_p).
    \]
    Applying \eqref{eq:sieve_upper} at level \(D_p\) (so \(s=1\)) yields
    \begin{align}
    S(\mathcal{A}_p,\mathcal{P},z)
    &\leq |\mathcal{A}_p| V(D_p)\left(F(1)+\varepsilon C_1(\varepsilon)e^2h(1)\right)
    +\sum_{\substack{d\mid P(z)\\ d < Q D_p}} |r_p(d)| \nonumber\\
    &= |\mathcal{A}_p| V(D_p)\left(2e^{\gamma}+\varepsilon C_1(\varepsilon)\right)
    +\sum_{\substack{d\mid P(z)\\ d < Q D_p}} |r_p(d)|, \label{eq:Sp_case2}
    \end{align}
    with \(r_p(d)\) as in \eqref{eq:rq_def}. Using Lemma~\ref{lem:mertens_products} once more,
    \begin{equation}\label{eq:VDp_bound}
    V(D_p)\leq \frac{e^{-\gamma}}{\log D_p}\left(1+\frac{1}{2(\log D_p)^2}\right)
    \leq \frac{k_1e^{-\gamma}}{s_p\log X}\left(1+\frac{1}{2(\log \widetilde D)^2}\right).
    \end{equation}
    Inserting \eqref{eq:VDp_bound} into \eqref{eq:Sp_case2}, and noting that \(D_p\geq D_y\), we find
    \begin{align}
    S(\mathcal{A}_p,\mathcal{P},z)
    &\leq \frac{k_1e^{-\gamma}}{\log X}\left(1+\frac{1}{2(\log \widetilde D)^2}\right)
    |\mathcal{A}_p|
    \left(\frac{2e^{\gamma}}{s_p}+\frac{\varepsilon C_1(\varepsilon)}{s_p}\right)
    +\sum_{\substack{d\mid P(z)\\ d< Q D_p}} |r_p(d)| \nonumber\\
    &= k_1e^{-\gamma}\left(1+\frac{1}{2(\log \widetilde D)^2}\right)
    |\mathcal{A}_p|
    \left(\frac{F(s_p)+\varepsilon C_1(\varepsilon)e^2h(s_p)}{\log X}\right)
    +\sum_{\substack{d\mid P(z)\\ d< Q D_p}} |r_p(d)|, \label{eq:Sp_case2_refined}
    \end{align}
    since \(F(s)=2e^\gamma/s\) and \(e^2h(s)=1/s\) at \(s=s_p\) in this range.
    
    \smallskip
    In summary, \eqref{eq:Sq_case1_refined} and \eqref{eq:Sp_case2_refined} provide the same upper bound for \(S(\mathcal A_p,\mathcal P,z)\) throughout \(z\leq p<y\). Weighting by \(\left(1-\log p/\log y\right)\) and summing over \(p\) gives
    \begin{align}
    \sum_{\substack{p\in\mathcal{P}\\ z\leq p<y}}
    &\left(1-\frac{\log p}{\log y}\right) S(\mathcal{A}_p,\mathcal{P},z) \nonumber\\
    &\leq k_1e^{-\gamma}\left(1+\frac{1}{2(\log \widetilde D)^2}\right)
    \sum_{\substack{p\in\mathcal{P}\\ z\leq p < y}}
    \left(1-\frac{\log p}{\log y}\right)
    |\mathcal{A}_p|
    \left(\frac{F(s_p)+\varepsilon C_1(\varepsilon)e^2 h(s_p)}{\log X}\right) \nonumber\\
    &\qquad
    +\sum_{\substack{p\in\mathcal{P}\\ z\leq p < y}}
    \left(1-\frac{\log p}{\log y}\right)
    \sum_{\substack{d\mid P(z)\\ d< Q D_p}} |r_p(d)|. \label{eq:weighted_split}
    \end{align}
    
    \smallskip
    \noindent\textbf{Bounding the remainder.}
    Because \(\mathcal A\) is an interval of consecutive integers, we have
    \[
    r_p(d)=|\mathcal{A}_{pd}|-\frac{|\mathcal{A}_p|}{d}
    \]
    satisfies \(|r_p(d)|\leq 1\). Consequently,
    \[
    \sum_{\substack{p\in\mathcal{P}\\ z\leq p < y}}
    \left(1-\frac{\log p}{\log y}\right)
    \sum_{\substack{d\mid P(z)\\ d< Q D_p}} |r_p(d)|
    \leq Q X^{\frac{1}{2}-\alpha}
    \sum_{\substack{p\in\mathcal{P}\\ z\leq p < y}}
    \left(1-\frac{\log p}{\log y}\right)\frac{1}{p}.
    \]
    
    We estimate the weighted reciprocal-prime sum via Lemma~\ref{lem:partial_summation_upper} together with Lemma~\ref{lem:vanlalngaia_prime_reciprocals}. In the notation of Lemma~\ref{lem:partial_summation_upper}, take
    \[
    g(x)=\log\log x,\qquad
    f(t)=1-\frac{\log t}{\log y},
    \]
    and
    \[
    c(n)=
    \begin{cases}
    \dfrac{1}{n}, & \text{if \(n\) is prime},\\[0.5ex]
    0, & \text{otherwise},
    \end{cases}
    \]
    with
    \[
    E=\frac{5k_1^3}{(\log X)^3}.
    \]
    (The choice of \(E\) comes from Lemma~\ref{lem:vanlalngaia_prime_reciprocals}.) This yields
    \begin{equation}\label{eq:weighted_recip_prime_sum}
    \sum_{\substack{p\in\mathcal{P}\\ z\leq p < y}}
    \left(1-\frac{\log p}{\log y}\right)\frac{1}{p}
    \leq \log\frac{k_1}{k_2}-1+\frac{k_2}{k_1}
    +\frac{5k_1^3}{(\log X)^3}\left(1-\frac{k_2}{k_1}\right).
    \end{equation}
    Therefore the remainder contribution in \eqref{eq:weighted_split} does not exceed \(\mathcal E(X)\), with \(\mathcal E(X)\) as in \eqref{eq:E_def}.

    \smallskip
    \noindent\textbf{Bounding the main term.}
    Again using that \(\mathcal A\) is a set of consecutive integers, we have
    \begin{equation}\label{eq:Ap_count}
    |\mathcal{A}_p|\leq \frac{2\sqrt{N}}{p}+1.
    \end{equation}
    Since \(k_1\leq 8\), we have \(0<s_p\leq 3\), and Lemma~\ref{lem:Ff_explicit} gives
    \begin{equation}\label{eq:Fsp_eval}
    F(s_p)=\frac{2e^\gamma}{s_p}=\frac{2e^\gamma\log X}{k_1\log D_p}.
    \end{equation}
    Combining \eqref{eq:Ap_count} and \eqref{eq:Fsp_eval} with the inequalities \(s_p\geq k_\alpha\), \(D_p\geq D_y\), and the monotonicity of \(h(s)\), we obtain
    \begin{align}\label{eq:main_term_decomp}
    &\sum_{\substack{p\in\mathcal{P}\\ z\leq p<y}}
    \left(1-\frac{\log p}{\log y}\right)|\mathcal{A}_p|
    \left(\frac{F(s_p)+\varepsilon C_1(\varepsilon)e^2 h(s_p)}{\log X}\right) \nonumber\\
    &\qquad \leq 2 \sqrt{N} \sum_{\substack{p\in\mathcal{P}\\ z\leq p<y}}\frac{1}{p}
    \left(1-\frac{k_2\log p}{\log X}\right)
    \left(\frac{2e^\gamma}{k_1\log D_p}
    +\frac{\varepsilon C_1(\varepsilon)e^2 h(k_\alpha)}{\log X}\right) \nonumber\\
    &\qquad \qquad
    +\sum_{\substack{p\in\mathcal{P}\\ z\leq p < y}}\left(1-\frac{k_2}{k_1}\right)
    \left(\frac{2e^\gamma}{k_1\log D_y}
    +\frac{\varepsilon C_1(\varepsilon)e^2 h(k_\alpha)}{\log X}\right).
    \end{align}
    
    The second term on the right-hand side is bounded by
    \[
    \frac{y}{\log X}\left(1-\frac{k_2}{k_1}\right)
    \left(\frac{2e^\gamma}{k_\alpha}+\varepsilon C_1(\varepsilon)e^2 h(k_\alpha)\right)
    = M_2(X),
    \]
    matching \eqref{eq:M2_def}.

    For the first term in \eqref{eq:main_term_decomp}, we apply Lemma~\ref{lem:partial_summation_upper} together with Lemma~\ref{lem:vanlalngaia_prime_reciprocals} in the same manner as above, but now with the weight
    \[
    f(t)=\frac{1}{\log D_t}-\frac{k_2\log t}{\log D_t\log X},
    \]
    so that the factor \(1/\log D_p\) is incorporated into the summation. This gives
    \begin{align*}
    &\sum_{\substack{p \in \mathcal P \\ z\leq p<y}}
    \left(1-\frac{k_2\log p}{\log X}\right)\frac{1}{p\log D_p} \\
    & \qquad\leq \frac{1}{(1-2\alpha)\log X}
    \Biggl(
    2 \log\frac{k_1}{k_2}
    +\left(k_2(1 - 2\alpha)-2\right)
    \log\frac{1-2\alpha-\frac{2}{k_2}}{1-2\alpha-\frac{2}{k_1}}
    \Biggr) \\
    & \qquad \qquad
    +\frac{5k_1^3}{\left(\frac12-\alpha-\frac{1}{k_1}\right)(\log X)^4}\left(1-\frac{k_2}{k_1}\right).
    \end{align*}
    Together with \eqref{eq:weighted_recip_prime_sum}, this yields that the first term in \eqref{eq:main_term_decomp} does not exceed \(M_1(X)\), where \(M_1(X)\) is as in \eqref{eq:M1_def}.

    Finally, inserting the bounds for the main term and the remainder term into \eqref{eq:weighted_split} gives \eqref{eq:weighted_sum_prop}.
\end{proof}

\section{Proof of Theorem~\ref{thm:main_omega3}}\label{sec:mainsect}

In this section we combine the results of Sections~\ref{sec:prelims} and~\ref{sec:Richertsect} to prove Theorem~\ref{thm:main_omega3}. Throughout, \(\mathcal{A}\) is as in \eqref{Adef} and \(X\) is as in \eqref{eq:X_def}. We first verify the condition \eqref{eq:Q0_cond} using two auxiliary estimates. We then fix the parameters \(k,k_1,k_2,c,\delta\), apply the sieve bounds from Section~\ref{sec:prelims} within the weighted-sieve framework of Section~\ref{sec:Richertsect}, and optimise numerically.

\begin{lemma}[{\cite[Lemma~7]{glasby2021most}}]\label{lem:glasby_tail}
For all \(a\geq 12\),
\[
\sum_{p\geq a}\frac{1}{p^2}\le \frac{2.22}{a\log a}.
\]
\end{lemma}

In our application we apply Lemma~\ref{lem:glasby_tail} only with \(a=z\), and our parameter choices ensure \(z>1000\), so the hypothesis \(a\geq 12\) is always satisfied.

\begin{lemma}[{\cite[Theorem~1]{rosser_schoenfeld1962}}]\label{lem:rosser_pi_upper}
    For \(x>1\),
    \[
    \pi(x)<\frac{x}{\log x}\left(1+\frac{3}{2\log x}\right),
    \]
    where \(\pi(x)\) denotes the prime-counting function.
\end{lemma}

We now prove Theorem~\ref{thm:main_omega3}. By Lemma~\ref{lem:small_n}, it is enough to treat the case \(N \geq 10^{31}\). Throughout this subsection we take
\[
k=3,\quad k_1=8,\quad \text{and} \quad  k_2=3.17.
\]
With these choices we verify \eqref{eq:Q0_cond} with
\[
c=0.297,\quad \delta=\frac14.
\]

\begin{lemma}\label{lem:Q0_omega3}
    Let \(N\geq 10^{31}\), \(z=X^{1/8}\), and \(y=X^{1/3.17}\). Then
    \[
    \sum_{\substack{p\in\mathcal P \\ z\leq p<y}} |\mathcal A_{p^2}|
    \leq 0.297\,|\mathcal A|^{3/4}.
    \]
    In particular, \eqref{eq:Q0_cond} holds with \(c=0.297\) and \(\delta=1/4\).
\end{lemma}

\begin{proof}
    Since \(\mathcal A\) is a set of consecutive integers, we have
    \begin{align}
    \sum_{\substack{p\in\mathcal P \\ z\leq p<y}} |\mathcal A_{p^2}|
    &\leq \sum_{\substack{p\in\mathcal P \\ z\leq p<y}}\left(\frac{|\mathcal A|}{p^2}+1\right) \notag\\
    &= |\mathcal A|\sum_{\substack{p\in\mathcal P \\ z\leq p<y}} \frac{1}{p^2}
    \;+\; \sum_{\substack{p\in\mathcal P \\ z\leq p<y}} 1 \notag\\
    &\leq \frac{2.22\,|\mathcal A|}{z\log z}+\pi(y),
    \label{eq:Q0_omega3_start}
    \end{align}
    where we used Lemma~\ref{lem:glasby_tail} to bound the sum over \(p^{-2}\).
    
    We first bound the term involving \(z\). Since \(X>N\), we have
    \[
    z=X^{1/8}\geq N^{1/8}\geq 7498.
    \]
    Also,
    \[
    z\geq N^{1/8}
    =2^{-1/4}(2\sqrt N)^{1/4}
    \geq 2^{-1/4}|\mathcal A|^{1/4}.
    \]
    Therefore,
    \begin{equation}\label{eq:Q0_omega3_zterm}
    \frac{2.22\,|\mathcal A|}{z\log z}
    \leq
    \frac{8 \cdot 2.22\cdot 2^{1/4}|\mathcal A|^{3/4}}{31\log 10}
    \leq 0.296\,|\mathcal A|^{3/4}.
    \end{equation}
    
    Next we bound \(\pi(y)\). We have
    \[
    y = X^{1/3.17} \geq N^{1/3.17} \geq 10^{31/3.17}.
    \]
    Also,
    \[
    y \leq (N + 2\sqrt{N})^{1/3.17} \leq 1.001 N^{1/3.17} \leq 0.0123\,|\mathcal{A}|^{3/4}.
    \]
    By Lemma~\ref{lem:rosser_pi_upper},
    \begin{equation}\label{eq:Q0_omega3_piterm}
        \pi(y) \leq 0.0523y \leq 6.5 \cdot 10^{-4}|\mathcal{A}|^{3/4}.
    \end{equation}
        
    Finally, substituting \eqref{eq:Q0_omega3_zterm} and \eqref{eq:Q0_omega3_piterm} into \eqref{eq:Q0_omega3_start}, we get
    \[
    \sum_{\substack{ p\in\mathcal P \\ z\leq p<y}} |\mathcal A_{p^2}|
    \leq 0.297\,|\mathcal A|^{3/4},
    \]
    as required.
\end{proof}

Next, we then apply the weighted-sieve framework from Section~\ref{sec:Richertsect}, together with the explicit sieve estimates from Section~\ref{sec:prelims}, to complete the proof for the remaining range \(N \geq 10^{31}\).

\begin{proof}[Proof of Theorem~\ref{thm:main_omega3}]
    Assume throughout that \(N \geq 10^{31}\). We apply Lemma~\ref{lem:Q0_omega3} and Lemma~\ref{lem:rk_lower_weighted} with
    \[
    k=3,\qquad z=X^{1/8},\qquad y=X^{1/3.17},
    \]
    so that
    \[
    \lambda=k+1-k_2=0.83.
    \]
    For \eqref{eq:rk_lower_main}, this gives
    \begin{equation}\label{eq:r3_lower_start}
    r_3(\mathcal A)
    \geq \frac{0.83}{3}S(\mathcal A,\mathcal P,z)
    -\frac{1}{3}\sum_{\substack{p\in\mathcal P \\ z\leq p<y}}
    \left(1-\frac{\log p}{\log y}\right)S(\mathcal A_p,\mathcal P,z)
    -\frac{0.83\cdot 0.297}{3}|\mathcal A|^{3/4}.
    \end{equation}
    We now bound each term on the right-hand side.
    
    First, we treat the \(Q_0(c,\delta)\)-remainder term. Since \(|\mathcal A|\leq 2\sqrt N\), we have
    \begin{equation}\label{eq:r3_q0_remainder_bound}
    \frac{0.83\cdot 0.297}{3}|\mathcal A|^{3/4}
    \leq 0.0014\frac{\sqrt N}{\log X}, 
    \end{equation}
    for all \(N \geq 10^{31}\).
    
    Next we obtain a lower bound for \(S(\mathcal A,\mathcal P,z)\). Let
    \[
    D=z^s,\qquad s=\frac{\log D}{\log z},
    \]
    where we assume \(s\in[3,4]\) and will be chosen later. Since \(z=X^{1/8}\) and \(N\geq 10^{31}\), we have \(z > 7498\), so Lemma~\ref{lem:linear_sieve} (with \(g(d)=1/d\)) together with Lemma~\ref{lem:mertens_products} gives
    \begin{align}
    S(\mathcal A,\mathcal P,z)
    &>|\mathcal A|V(z)\Bigl(f(s)-\varepsilon C_2(\varepsilon)e^2h(s)\Bigr)
    -\sum_{d<QD}\mu^2(d) \notag\\
    &\geq \frac{e^{-\gamma}(2\sqrt N-1)}{\log z}
    \left(1-\frac{1}{2(\log z)^2}\right)
    \Bigl(f(s)-\varepsilon C_2(\varepsilon)e^2h(s)\Bigr)
    -\sum_{d<QD}\mu^2(d),
    \label{eq:SAz_lower_1}
    \end{align}
    where \(\varepsilon = 1.13\cdot10^{-3}\) (via Lemma~\ref{lem:mertens_explicit_bounds}) and \(C_2(\varepsilon) = 116\) from \cite[Table~1]{bordignon_johnston_starichkova2025}.
    
    Since \(s\in[3,4]\), Lemma~\ref{lem:Ff_explicit} and \eqref{eq:h_def} give
    \[
    f(s)=\frac{2e^\gamma\log(s-1)}{s},
    \qquad
    h(s)=3s^{-1}e^{-s}.
    \]
    It is therefore convenient to define
    \begin{equation}\label{eq:Cs_def_omega3}
        f(s)-\varepsilon C_2(\varepsilon)e^2h(s) > C(s):=\frac{2e^\gamma\log(s-1)-0.39324e^{2-s}}{s}.
    \end{equation}
    
    To bound the square-free remainder term, we use the explicit estimate (see \cite[Lemma~4.6]{ramare2013mobius})
    \[
    \sum_{d\leq x}\mu^2(d)\leq \frac{6}{\pi^2}x+0.5\sqrt x\qquad (x\geq 10).
    \]
    For \(D = z^s \geq z^3 \geq 4.21 \cdot 10^{11}\), this implies
    \begin{equation}\label{eq:mu_square-free_remainder}
    \sum_{d<2D}\mu^2(d)\leq 1.22\,D 
    =1.22\,X^{s/8}.
    \end{equation}
    Substituting \eqref{eq:Cs_def_omega3} and \eqref{eq:mu_square-free_remainder} into \eqref{eq:SAz_lower_1}, we obtain
    \begin{equation}\label{eq:SAz_lower_2}
    S(\mathcal A,\mathcal P,z)
    >
    \frac{e^{-\gamma}(2\sqrt N-1)}{\log z}
    \left(1-\frac{1}{2(\log z)^2}\right)C(s)
    -1.22 X^{s/8}.
    \end{equation}
    
    We now choose \(s=3.33\), for which \(C(s)>0.873\). A direct calculation then gives 
    \begin{equation}\label{eq:SAz_final_lower}
        \frac{0.83}{3}S(\mathcal A,\mathcal P,z) > 2.095\frac{\sqrt N}{\log X}. 
    \end{equation}
    
    It remains to bound the weighted sum in \eqref{eq:r3_lower_start}. We apply Proposition~\ref{prop:weighted_sifting_bound} with \(\alpha=0.06\). For this choice of \(\alpha\), one checks that
    \[
    D_y=X^{\frac{1}{2}-0.06-\frac{1}{3.17}}>7258,
    \]
    so that the product condition in Proposition~\ref{prop:weighted_sifting_bound} holds with \(Q =2\), \(\varepsilon=1.13 \cdot 10^{-3}\) (from Lemma~\ref{lem:mertens_explicit_bounds}) and \(C_1(\varepsilon) = 115\) (from \cite[Table~1]{bordignon_johnston_starichkova2025}).
    
    Hence Proposition~\ref{prop:weighted_sifting_bound} yields
    \begin{align*}
    \frac{1}{3}\sum_{\substack{p\in\mathcal P\\ z\leq p<y}}
    \left(1-\frac{\log p}{\log y}\right)S(\mathcal A_p,\mathcal P,z)
    &\leq
    \frac{8 e^{-\gamma}}{3}
    \left(1+\frac{1}{2(\log D_y)^2}\right)\left(M_1(X)+M_2(X)\right)
    +\frac{\mathcal E(X)}{3}.
    \end{align*}
    Using the definitions of \(M_1(X)\), \(M_2(X)\), and \(\mathcal E(X)\), together with the above parameter choices, we compute
    \[
    \frac{8 e^{-\gamma}}{3} \left(1+\frac{1}{2(\log D_y)^2}\right) \leq 1.507, 
    \]
    \[
    M_1(X)
    \leq 1.230\frac{\sqrt N}{\log X}, 
    \]
    \[
    M_2(X)
    \leq 4.6 \cdot10^{-6}\frac{\sqrt N}{\log X}, 
    \]
    and
    \[
    \frac{\mathcal E(X)}{3}\leq 0.215\frac{\sqrt N}{\log X}. 
    \]
    Therefore,
    \begin{equation}\label{eq:weighted_sum_final_upper}
    \frac{1}{3}\sum_{\substack{p\in\mathcal P\\ z\leq p<y}}
    \left(1-\frac{\log p}{\log y}\right)S(\mathcal A_p,\mathcal P,z)
    \leq 2.0687\frac{\sqrt N}{\log X}. 
    \end{equation}
    
    Finally, substituting \eqref{eq:r3_q0_remainder_bound}, \eqref{eq:SAz_final_lower}, and \eqref{eq:weighted_sum_final_upper} into \eqref{eq:r3_lower_start}, we obtain
    \[
    r_3(\mathcal A)> 0.0249 \frac{\sqrt N}{\log X}>0.
    \]
    Thus \(r_3(\mathcal A) > 0\) when \(N\geq 10^{31}\), and so there exists \(a\in\mathcal A\) with \(\Omega(a)\leq 3\). Since \(\mathcal A=\mathcal A(N)\subset (n^2,(n+1)^2)\), this proves Theorem~\ref{thm:main_omega3}.
\end{proof}

\section*{Acknowledgements}

The author thanks Jonathan Sorenson for suggesting the idea underlying Lemma~\ref{lem:small_n}, Adrian Dudek for careful proofreading and helpful suggestions, and Daniel Johnston for valuable comments and suggestions that improved the paper. He is also grateful to the anonymous referee for carefully reading the manuscript and for many helpful comments that improved the exposition. This research was supported by the Commonwealth through an Australian Government Research Training Program (RTP) Scholarship.

\bibliographystyle{abbrv}
\bibliography{refs}

@article{pintz2009landau,
    author  = {Pintz, J{\'a}nos},
    title   = {{Landau's Problems on Primes}},
    journal = {J. Théor. Nombres Bordeaux},
    year    = {2009},
    volume  = {21},
    number  = {2},
    pages   = {357--404},
}

@article{baker_harman_pintz2001,
    author  = {Baker, R. C. and Harman, G. and Pintz, J{\'a}nos},
    title   = {{The Difference Between Consecutive Primes, {II}}},
    journal = {Proc. Lond. Math. Soc.},
    year    = {2001},
    volume  = {83},
    number  = {3},
    pages   = {532--562},
}

@book{brun1920crible,
    author    = {Brun, Viggo},
    title     = {{Le crible d'{\'E}ratosth{\`e}ne et le th{\'e}or{\`e}me de Goldbach}},
    year      = {1920},
    series    = {Skrifter utgit av Videnskapsselskapet i Kristiania. I. Matematisk-naturvidenskabelig Klasse},
    number    = {3},
    address   = {Kristiania},
    publisher = {J. Dybwad},
    pagetotal = {36},
}

@article{chen1975almostprimesinterval,
    author   = {Chen, Jing-Run},
    title    = {{On the Distribution of Almost Primes in an Interval}},
    journal  = {Scientia Sinica},
    year     = {1975},
    volume   = {18},
    pages    = {611--627},
    zbl      = {0381.10033},
    mrnumber = {MR05615584},
}

@article{dudek_johnston2026apba,
    author  = {Dudek, Adrian W. and Johnston, Daniel R.},
    title   = {{Almost Primes Between All Squares}},
    journal = {J. Number Theory},
    year    = {2026},
    volume  = {278},
    pages   = {726--745},
    doi     = {10.1016/j.jnt.2025.05.009},
}

@article{bordignon_johnston_starichkova2025,
    author  = {Bordignon, Matteo and Johnston, Daniel R. and Starichkova, Valeriia},
    title   = {{An Explicit Version of {Chen}'s Theorem and the Linear Sieve}},
    journal = {Int. J. Number Theory},
    year    = {2025},
    volume  = {21},
    number  = {10},
    pages   = {2497--2572},
    doi     = {10.1142/S1793042125501192},
}

@inproceedings{kuhn1954neue_abschaetzungen,
    author    = {Kuhn, P.},
    title     = {{Neue Absch{\"a}tzungen auf Grund der Viggo Brunschen Siebmethode}},
    booktitle = {Proceedings of the 12th Scandinavian Mathematical Congress (Lund, 1953)},
    year      = {1954},
    pages     = {160--168},
}

@article{richert1969selberg_weights,
    author = {Richert, Hans-Egon},
    title   = {{Selberg's Sieve with Weights}},
    journal = {Mathematika},
    year    = {1969},
    volume  = {16},
    number  = {1},
    pages   = {1--22},
}

@misc{johnston2026primesprimescubes,
    author = {Johnston, Daniel R. and Sorenson, Jonathan P. and Thomas, Simon N. and Webster, Jonathan E.},
    title  = {{Primes and Almost Primes Between Cubes}},
    year   = {2026},
    note   = {Preprint, arXiv:2601.15564},
}

@article{sorenson_webster_2025_legendre,
    author  = {Sorenson, Jonathan and Webster, Jonathan},
    title   = {{An Algorithm and Computation to Verify Legendre's Conjecture up to {$7\cdot 10^{13}$}}},
    journal = {Res. Number Theory},
    year    = {2025},
    volume  = {11},
    number  = {1},
    pages   = {4},
    doi     = {10.1007/s40993-024-00589-4},
}

@misc{pzktupel_risinggap,
  author       = {Nicely, Thomas R.},
  title        = {{Prime Number Gap Record Rising}},
  howpublished = {\url{https://www.pzktupel.de/RecordGaps/risinggap.php}},
  note         = {Record table with contributor attributions. Last updated 16 February 2026. Accessed 28 February 2026},
  year         = {2026}
}

@article{cai2008chen2,
    author  = {Cai, Yingchun},
    title   = {{A Remark on {Chen}'s Theorem ({II})}},
    journal = {Chinese Ann. Math. Ser. B},
    volume  = {29},
    number  = {6},
    year    = {2008},
    pages   = {687--698},
}

@article{rosser_schoenfeld1962,
    author  = {Rosser, J. Barkley and Schoenfeld, Lowell},
    title   = {{Approximate Formulas for Some Functions of Prime Numbers}},
    journal = {Ill. J. Math.},
    volume  = {6},
    number  = {1},
    year    = {1962},
    pages   = {64--94},
}

@misc{mertensbounds_github,
    author       = {Johnston, Daniel R.},
    title        = {{MertenBounds}},
    howpublished = {\url{https://github.com/DJmath1729/MertenBounds}},
    note         = {GitHub repository. Accessed 28 February 2026},
    year         = {2026}
}

@book{greaves2001sieves,
    author    = {Greaves, George},
    title     = {{Sieves in Number Theory}},
    series    = {Ergebnisse der Mathematik und ihrer Grenzgebiete. 3. Folge / A Series of Modern Surveys in Mathematics},
    volume    = {43},
    publisher = {Springer},
    year      = {2001},
    isbn      = {9783540416470},
}

@article{vanlalngaia2017explicit,
    author  = {Vanlalngaia, Ramdinmawia},
    title   = {{Explicit Mertens Sums}},
    journal = {Integers},
    volume  = {17},
    year    = {2017},
    pages   = {A11}
}

@article{ramare2013mobius,
    author  = {Ramar{\'e}, Olivier},
    title   = {{From Explicit Estimates for Primes to Explicit Estimates for the {M{\"o}bius} Function}},
    journal = {Acta Arith.},
    volume  = {157},
    number  = {4},
    year    = {2013},
    pages   = {365--379},
    doi     = {10.4064/aa157-4-4},
}

@article{glasby2021most,
    author  = {Glasby, S. P. and Praeger, Cheryl E. and Unger, W. R.},
    title   = {{Most Permutations Power to a Cycle of Small Prime Length}},
    journal = {Proc. Edinb. Math. Soc.},
    year    = {2021},
    volume  = {64},
    number  = {2},
    pages   = {234--246},
}

@article{iwaniec1980new,
    author  = {Iwaniec, Henryk},
    title   = {{A New Form of the Error Term in the Linear Sieve}},
    journal = {Acta Arith.},
    year    = {1980},
    volume  = {37},
    number  = {1},
    pages   = {307--320},
}
\end{document}